\documentclass[12pt]{amsart}

\usepackage{enumerate}
\usepackage{graphicx}
\usepackage{amsmath}
\usepackage{amssymb}
\usepackage{amsfonts}
\usepackage[dvipsnames]{xcolor}
\usepackage{hyperref}
\usepackage{tikz}
\usepackage{float}
\allowdisplaybreaks
 \newtheorem{theorem}{Theorem}[section]
 \newtheorem{corollary}[theorem]{Corollary}
 \newtheorem{lemma}[theorem]{Lemma}

\theoremstyle{definition}

\theoremstyle{remark}

\newtheorem{fact*}{Fact}

\newcommand{\hilbert}{\mathcal{H}}

\newcommand{\G}{\mathcal{G}}

\newcommand{\BH}{\mathcal{B}(\mathcal{H})}

\newcommand{\R}{\mathbb{R}}

\newcommand{\norm}[1]{\left\Vert#1\right\Vert}

\newcommand{\ip}[2]{\left\langle #1, #2 \right\rangle}
\newcommand{\ad}{^\ast}
\newcommand{\inv}{^{-1}}

\newcommand{\til}{\raise.17ex\hbox{$\scriptstyle\mathtt{\sim}$}}

\newcommand\beq{\begin{equation}}

\newcommand\eeq{\end{equation}}

\newcommand{\bbm}{\left[ \begin{smallmatrix}}
\newcommand{\ebm}{\end{smallmatrix} \right]}
\newcommand{\bpm}{\begin{pmatrix}}
\newcommand{\epm}{\end{pmatrix}}
\numberwithin{equation}{section}

\newlength{\Mheight}
\newlength{\cwidth}

\newcommand{\dfn}[1]{{\bf #1}\index{#1}}

\title[Graphical functions]{Matrix convex verbatim enumeration functions are graphical}
\author[J. E. Pascoe]{
J. E. Pascoe$^\dagger$
}
\address{Department of Mathematics\\
Drexel University\\
3141 Chestnut St \\
 Philadelphia, PA 19104}
\email[J. E. Pascoe]{james.eldred.pascoe@drexel.edu}
\thanks{$\dagger$ Partially supported by National Science Foundation DMS Analysis Grant 2319010}

\author[R. Tully-Doyle]{
Ryan Tully-Doyle$^\ddagger$
}
\address{Mathematics Department \\
Cal Poly, SLO\\
1 Grand Ave \\
San Luis Obispo, CA 93407}
\email[R. Tully-Doyle]{rtullydo@calpoly.edu}
\thanks{$\ddagger$ Partially supported by National Science Foundation DMS Analysis Grant 2055098}

\date{\today}

\subjclass[2020]{15A09, 47A56}
\keywords{}

\begin{document}

\begin{abstract}
We give a relation between verbatim generating functions of what we call Pythagorean languages and matrix convexity. Namely, several multivariate matrix convex functions occurring in the existing matrix analysis literature arise naturally in a combinatorial way. 

We give a Gelfand type formula for the numerical radius.
\end{abstract}

\maketitle

\section{Introduction}

Let $\G$ be an alphabet consisting of formal symbols and equipped with an involution $*$ taking $a\in G \mapsto a\ad$. If $a = a\ad$, then $a$ is a \dfn{self-adjoint} element. Otherwise, we refer to $a$ as an \dfn{analytic} element. We can extend the involution on $\G$ to the collection of all words in $G$ in the natural way:
\[
(a_1 a_2 a_3 \ldots a_g)\ad = a_g\ad \ldots a_3\ad a_2\ad a_1 \ad
\]
A \dfn{formal language} $\mathcal{L}$ is a collection of words in $\G$. We call a language \dfn{self-adjoint} if it is closed under the $*$ operation. A self-adjoint language $\mathcal L$ is \dfn{Pythagorean} if for nonempty words $\alpha, \beta, \gamma$, whenever $\beta\ad \alpha $ and $\gamma\ad \alpha$ are in $\mathcal L$ then $\alpha\ad \alpha$ and $\gamma\ad \beta$ are in $\mathcal L$. 

The set of \dfn{ellipses} of $\mathcal L$, denoted $E_{\mathcal{L}}$, is the set of words $\alpha$ in $G$ such that there exists a word $\beta$ in $G$ so that $\beta\ad \alpha \in \mathcal L$. Two ellipses $\alpha, \beta$ are \dfn{morphologically equivalent} if $\beta\ad \alpha \in \mathcal L$, which we denote $\alpha \cong \beta$. Note for a Pythagorean language, this is an equivalence relation on $E_\mathcal{L}$. We say there is an \dfn{$a_i$-edge} from $\alpha \in E_\mathcal{L}$ to $\beta \in E_\mathcal{L}$ if $\beta^* a_i \alpha \in \mathcal{L}.$ (The edge is undirected if $a_i$ is self-adjoint and directed otherwise.) That is, we get a natural graph structure for each symbol on the set of ellipses. 

Examples include the set of balanced brackets (e.g. ``$[[]][]$'') and the set of irreducible balanced brackets (e.g. ``$[[[]][]]$'', which cannot be written as a product of smaller brackets).  Note that the previous examples can be naturally understood as Dyck paths.
\begin{figure}[htbp]
  \centering
  \begin{tikzpicture}[scale=0.5]
    \draw[gray,very thin] (0,0) grid (10,5);
    \draw[->] (0,0) -- (10,0) node[right] {$x$};
    \draw[->] (0,0) -- (0,5) node[above] {$y$};
  
    \draw[ultra thick, red] (0,0) -- (1,1) -- (2,2) -- (3,3) -- (4,2) -- (5,1) -- (6,2) -- (7,3) -- (8,2) -- (9,1) -- (10,0);
    
  \end{tikzpicture}
  \caption{Example of a Dyck path corresponding to the string ``$[[[]][[]]]$''.}
  \label{fig:dyckpath}
\end{figure}
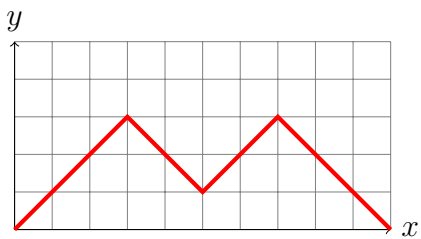
A Dyck path is a lattice path in the upper half-plane consisting of unit steps up or down in the $y$ direction, starting and ending at the same $y=0$, and never crossing below the x-axis.

 Irreducible Motzkin paths can be expressed similarly with the addition of a self-adjoint symbol $x$ (e.g. ``$[[x]x[]]$'').
 
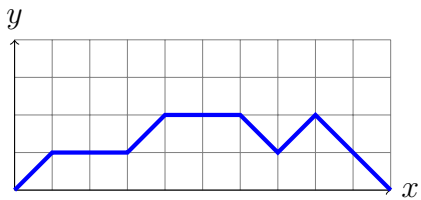
\begin{figure}[htbp]
  \centering
  \begin{tikzpicture}[scale=0.5]
    \draw[gray,very thin] (0,0) grid (10,4);
    \draw[->] (0,0) -- (10,0) node[right] {$x$};
    \draw[->] (0,0) -- (0,4) node[above] {$y$};
  
    \draw[ultra thick, blue] (0,0) -- (1,1) -- (2,1) -- (3,1) -- (4,2) -- (5,2) -- (6,2) -- (7,1) -- (8,2) -- (9,1) -- (10,0);
    
  \end{tikzpicture}
  \caption{Example of a Motzkin path, which corresponds to ``$[xx[xx][]]$'' in our string formulation.}
  \label{fig:motzkinpath}
\end{figure}
A Motzkin path is a lattice path that consists of unit steps up and down in the $y$-direction, as well as horizontal steps with no change in the $y$-coordinate such that the path starts and ends on the $x$-axis, and it never goes below the $x$-axis. We will discuss these examples and more in Section \ref{sec_examples}.

The \dfn{verbatim enumeration function} of $\mathcal L$ is the formal sum over all words in the language
\[
f(a_1, \ldots, a_g) = \sum_{\omega \in \mathcal{L}}\omega.
\]
We show that the verbatim enumeration function $f$ corresponding to a Pythagorean language $\mathcal L$ is matrix convex. A corresponding representation of $f$, the so-called butterfly realization, arises from a natural (directed) graph with chromatic edges. We will also see that many important functions arise this way, including the Ando unitarization function arising in the proof of Crouzeix's conjecture for the disk \cite{bcd, ando73}, the Anderson-Morley-Trapp function \cite{amt}, and, more obliquely, the matrix geometric mean \cite{ando94, pusz, palfiaadv}. Such functions also provide important examples in the study of boundary regularity in several complex variables, as in \cite{amchvms, bps1, bickel2020level, kn08ub, knese15, pascoePEMS, tdopmat}.

\section{Matrix convex functions}

Verbatim enumeration functions can be evaluated by replacing the indeterminants with appropriate matrices, so that self-adjoint matrices are substituted into self-adjoint indeterminants. This is part of the subject of free noncommutative analysis (see \cite{vvw12} for a comprehensive survey). 

We are concerned with such functions that are \dfn{matrix convex}. The particular approach we take here, a consequence of the so-called royal road theorem, is considered in great generality in \cite{royal}. 
 Let $D$ be the neighborhood of $0$ defined by
\[
D = \{ A: \norm{A} \leq r\}
\]
where $A$ is a $d$-tuple of matrices of the same size, and the norm is the operator norm on tuples of matrices (that is, $D$ consists of $d$-tuples of $n \times n$ matrices for all integers $n$. (This is an example of a convex free set whose study was intiated in \cite{effwink}.) For a function $f$ that maps the self-adjoint tuples in $D$ into self-adjoint matrices, we say that $f$ is matrix convex on $D$ if 
\[
f\left(\frac{A + B}{2}\right) \leq \frac{f(A) + f(B)}{2}
\]
for all $A, B$ of the same size in $D$. (Here we say  that $E\leq F$ for $n$ by $n$ matrices $E,F$ if $F-E$ is positive semidefinite. This is sometimes called the \dfn{L\"owner ordering}.) A consequence of the royal road theorem is if $f$ is matrix convex on $D$, then $f$ has a convergent power series on $D$ given by
\[
f(X) = \sum_{\omega} c_\omega X^\omega,
\]
where $\omega$ is words in $a_1, \ldots, a_d$, and the entries $A_i$ are substituted into the appropriate $a_i$. 

Matrix convex functions have formulas in terms of Hilbert space objects called \dfn{realizations}. In one variable, this is the classical Kraus representation, which says that a matrix convex function $f:[-1,1]\to \R$ has the form 
\[
f(x) = a + bx + \int_{[-1,1]} \frac{x^2}{1 + tx} \, d\mu(x) 
\]
for constants $a \in \R, b \geq 0$ and a finite Borel measure $\mu$ supported on $[-1,1]$. Matrix convex functions in several variables possess a \dfn{butterfly realization} (see \cite{heltonbutterfly, royal}). The following theorem is a version of Theorem 4.4 in  \cite{royal}. (The point in \cite{royal, heltonbutterfly} being that if one formally takes the Hessian of the butterfly realization below, that one obtains a sum of Hermitian squares which is thus obviously positive, and thus functions of such a form can be viewed as obviously matrix convex.)

\begin{theorem}[\cite{royal}]
Let $f$ be matrix convex with a power series converging absolutely and uniformly on a neighborhood $D$ of $0$. Then there exist contractions $T_i$, vectors $Q_i$, a scalar $a_0$, and a continuous linear function $L$ so that 
\beq\label{eq_kraus}
f(X) = a_0 + L(X) + (\sum Q_i X_i)\ad(I - \sum T_i X_i)\inv (\sum Q_i X_i)
\eeq
\end{theorem}
We note that $Q_i\ad T^\alpha Q_j$ is equal to the coefficient of the word $a_i^*a^\alpha a_j.$ Note that
this implies that \[
C = \left[ c_{\beta\ad \alpha} \right]_{\alpha, \beta} \geq 0
\] is positive semidefinite. Our goal will be to see that when such a positive semidefinite $C$ is a $0-1$ matrix-- they will turn out to the be generating functions of Pythagorean language.

Matrix convexity is used to establish the positivity of a Hankel matrix, which in turn is used to derive the butterfly realization formula. We note that the reverse implication follows from algebraically expanding out
the butterfly realization, although that is not totally explicit in \cite{royal}.
\begin{lemma}[\cite{royal}]
A function $f$ is matrix convex on $D$ if and only if the block matrix 
\[
C = \left[ c_{\beta\ad \alpha} \right]_{\alpha, \beta} \geq 0,
\]
where $\alpha, \beta$ range over words of positive length.
\end{lemma}

We begin by establishing that a Hankel-type matrix consisting of the coefficients of the function $f$ that enumerates a Pythagorean language $\mathcal L$ is positive semi-definite. Note that we derive the positivity directly (without the assumption of matrix positivity).

\begin{lemma}\label{lemma_pos}
Let $\mathcal L$ be a formal language. The infinite matrix
\[
C = [c_{\beta\ad\alpha}]_{\alpha, \beta \in \mathcal L} = [1_L(\beta\ad\alpha)]_{\alpha, \beta \in \mathcal L} \geq 0
\]
if and only if
$\mathcal{L}$ is Pythagorean.
\end{lemma}
\begin{proof} 

Let $V$ be the quotient of $E_{\mathcal{L}}$ induced by the equivalence relation $\cong$. Let $\hilbert$ be the Hilbert space with elements of $V$ comprising an orthonormal basis and extending by linearity, with the typical $L^2$ inner product. Given $\alpha, \beta \in \mathcal L$, notice that $\ip{\alpha}{\beta} = 1$ if $\alpha \cong \beta$ and $0$ otherwise.

Let $f$ be the verbatim enumeration function of $\mathcal L$ with formal series
\[
f = \sum_\omega \omega.
\]

Let $c_{\beta\ad\alpha}$ be the coefficient of $\beta\ad\alpha$ term in the expansion of $f$, which is 1 if $\beta\ad\alpha \in \mathcal L$ and $0$ otherwise; that is, $c_{\beta\ad\alpha} =1$ precisely when $a \cong b$. Thus we can view $c_{\beta\ad\alpha} = \ip{\alpha}{\beta}_\hilbert$, and so $[c_{\beta\ad\alpha}]_{\alpha, \beta}$ has a Gram representation in terms of the basis elements in $V$. We conclude that $[c_{\beta\ad\alpha}]_{\alpha,\beta} \geq 0$.
\end{proof}

\begin{theorem}
Let $\mathcal L$ be a formal  language. The the verbatim enumeration function $f$ of $\mathcal L$ is matrix convex if and only if $\mathcal{L}$ is Pythagorean.
\end{theorem}

\begin{proof}
Essentially, the proof is as in \cite[Theorem 4.4]{royal}, where the positivity of $[c_{\beta\ad\alpha}]$ is used to construct a Kraus representation. Lemma \ref{lemma_pos} allows the same argument to be executed without the assumption of matrix convexity, which implies that $f$ has a Kraus representation \eqref{eq_kraus}. The structure of the representation directly implies that $f$ is matrix convex.
\end{proof}

We note that the butterfly realizations of such functions have $T_i$ given by the adjacency matrix/operator for the $a_i$-edges.
\section{Some examples}\label{sec_examples}
The simplest example is the Schur complement of a matrix, which, it goes without saying, arises ubiquitously. We illustrate the graph of the Schur complement in Figure \ref{figschur}.
\begin{figure}[H]
\centering
\begin{tikzpicture}[scale=2]
  \tikzset{
    vertex/.style = {circle, draw, fill=black, inner sep=0pt, minimum size=4pt},
  }

  \node[vertex] (v) at (0,0) {};
  \draw (v) to [out=120,in=30,looseness=100] (v);

\end{tikzpicture}
\caption{The undirected loop graph on one vertex corresponds to the Schur complement of $\bbm A & B \\ C & D -1  \ebm$, that is, $A-B(D-1)^{-1}C$, where the edge is $D$-colored.}
\label{figschur}
\end{figure}
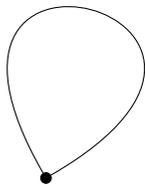

A second example arises in the work of Ando on numerical range. The central result in \cite{ando73} is that for any matrix $Z$ with numerical radius less than or equal to $\frac{1}{2}$, there exists a maximum matrix $Y$ such that
	$$\bpm 1-Y & Z \\ Z^* & Y \epm \geq 0$$
which is then used to obtain dilation theoretic type results. This example is illustrated in Figure \ref{figando}.
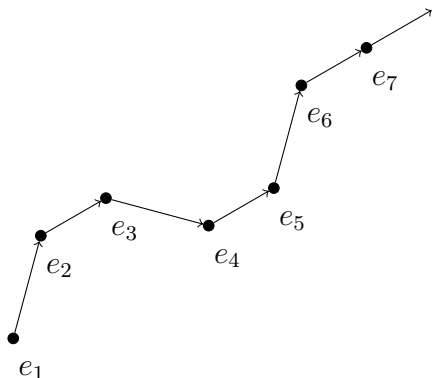
\begin{figure}[H]
\centering
\begin{tikzpicture}[rotate=30]
  \tikzset{
    vertex/.style = {circle, draw, fill=black, inner sep=0pt, minimum size=4pt},
  }

  \foreach \x in {0, 1, ..., 6} {
    \pgfmathtruncatemacro{\label}{\x + 1}
    \pgfmathtruncatemacro{\y}{mod(\x, 4) == 0 ? 0 : mod(\x, 4) == 1 ? 1 : mod(\x, 4) == 2 ? 0 : -1}
    \node[vertex] (v\x) at (\x, \y) {};
    \node at (\x, \y-0.5) {$e_{\label}$};
  }
  
  \foreach \x [evaluate=\x as \nextx using int(\x+1)] in {0, 1, ..., 5} {
    \draw[->] (v\x) -- (v\nextx);
  }

  \draw[->] (v6) -- ++(1, 0);


\end{tikzpicture}
\caption{An infinite directed path graph corresponds to Ando's unitarization function used in analysis of the numerical radius. The Ando unitarization function satisfies $Y(Z) = Z(1-Y(Z))^{-1}Z^*.$ Note that the verbatim enumeration function gives the irreducible Dyck paths.}
\label{figando}
\end{figure}

In \cite{amt}, a vast generalization was considered. Anderson, Morley and Trapp showed that there is a function $Y(X)$ such that, given a block matrix $X=(X_{ij})_{ij},$ we have that 
$X+D\otimes Y(X)$ is positive semidefinite, D is the matrix with a $-1$ in the upper corner and $1$ on the rest of the diagonal and $Y(X)$ is the maximum solution. 
For example, in the two by two case, we seek the maximal solution $Y$ to
$$\bpm 1-Y-X & Z \\ Z^* & Y  \epm \geq 0$$
for inputs $X$ and $Z.$, where we have normalized $X_{22}=0$ and relabeled the variables to make things clearer. Figure \ref{figtrapp} describes how we may obtain the function.

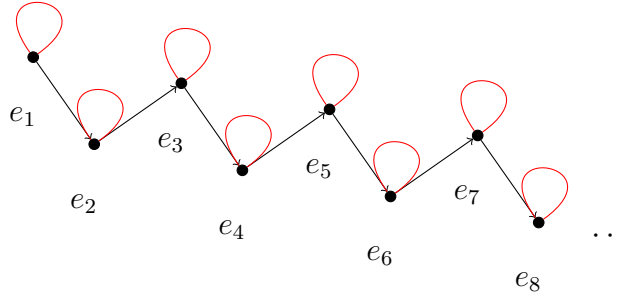
\begin{figure}[H]
\centering
\begin{tikzpicture}[rotate=-10]

  \tikzset{
    vertex/.style = {circle, draw, fill=black, inner sep=0pt, minimum size=4pt},
  }

  \foreach \x [evaluate=\x as \nextx using int(\x+1)] in {0, 1, ..., 7} {
    \pgfmathtruncatemacro{\label}{\x + 1}
    \pgfmathtruncatemacro{\y}{mod(\x, 2) == 0 ? 1 : 0}
    \node[vertex] (v\x) at (\x, \y) {};
    \node at (\x, \y-0.8) {$e_{\label}$};
  }

  \foreach \x [evaluate=\x as \nextx using int(\x+1)] in {0, 1, ..., 6} {
    \draw[->] (v\x) -- (v\nextx);
  }
  
  \foreach \x in {0, 1, ..., 7} {
    \draw[red] (v\x) to [out=135,in=45,looseness=30] (v\x);
  }

  \node at (8, 0) {$\cdots$};

\end{tikzpicture}
\caption{The two by two Anderson-Morley-Trapp function corresponds to an infinite path graph with self-loops added at each of the vertices. Note the verbatim enumeration function gives Motzkin paths. The Anderson-Morley-Trapp function satisfies $Y(Z,X) = Z(1-Y(Z,X)-X)^{-1}Z^*.$ In general, for larger sizes, we see the nodes are blown up into complete graphs.}
\label{figtrapp}
\end{figure}

\section{A Gelfand-type numerical radius formula}

Ando's unitarization function satisfies 
\[
Y(Z) = Z(1-Y(Z))^{-1}Z^*.
\]

Note that if $Y$ and $1-Y$ are invertible, then
\[
V = Y^{-1/2} Z (1-Y)^{-1/2}
\]
is an isometry. The following result is due to Ando \cite{ando73}. 

\begin{theorem}
Let $Z$ be in $\BH$. The numerical radius of $Z$ is less than or equal to $1/2$ if and only if
there exists an isometry $V$ and a positive contraction $Y$ such that 
\[
Z = Y^{1/2} V (1- Y)^{1/2}.
\]
\end{theorem}

As discussed in Figure \ref{figando}, $Y(Z)$ is the verbatim enumeration function for the irreducible Dyck paths. Now write 
\[
Y(Z) = \sum_{n} p_n(Z, Z\ad)
\]
where $p_n$ is a homogenous polynomial of degree $2n$. (Note $p_n$ enumerates the irreducible Dyck paths of length $2n$.) We see the following corollary of the main result concerning the numerical radius of a matrix $Z$ (see e.g. \cite{goldberg, tsatnum}).

\begin{corollary}
Let $Z \in \BH.$ The numerical radius of $Z$ is equal to 
	\beq\label{gelfrad}
		\frac{1}{2}\limsup_{n\rightarrow \infty} \|p_n(Z,Z\ad)\|^{1/2n}.
	\eeq
\end{corollary}

The first few polynomials $p_n$ are given by
\begin{align*}
p_2(Z, Z\ad) &= Z Z\ad \\
p_4(Z, Z\ad) &= Z Z Z\ad Z\ad \\
p_6(Z, Z\ad) &= Z Z Z Z\ad Z\ad Z\ad + Z Z Z\ad Z Z\ad Z\ad \\
p_8(Z, Z\ad) &= Z Z Z Z Z\ad Z\ad Z\ad Z\ad + Z Z Z Z\ad Z Z\ad Z\ad Z\ad \\
& \hspace{1cm}+ Z Z Z Z\ad Z\ad Z Z\ad Z\ad + Z Z Z\ad Z Z\ad Z Z\ad Z\ad + Z Z Z\ad Z Z Z\ad Z\ad Z\ad.
\end{align*}

For example, consider the matrix
\[
A = \bbm 0& 2 & 2 \\ 3 & 1 & 0 \\ 1 & 0 & 1\ebm.
\]
Applying \eqref{gelfrad}, numerical computations give that
\[
\frac{1}{2}  \| p_{1500}(A, A\ad) \|^{1/3000} \approx 3.445,
\]
with a convergence plot for values of $n$ up to $500$ shown in Figure \ref{fig:convplot}. 
Numerically solving using optimization gives
\[
\max_{\|v\|=1} \ip{A v}{v} \approx 3.458.
\]

\begin{figure}[htbp]
  \includegraphics[scale=.4]{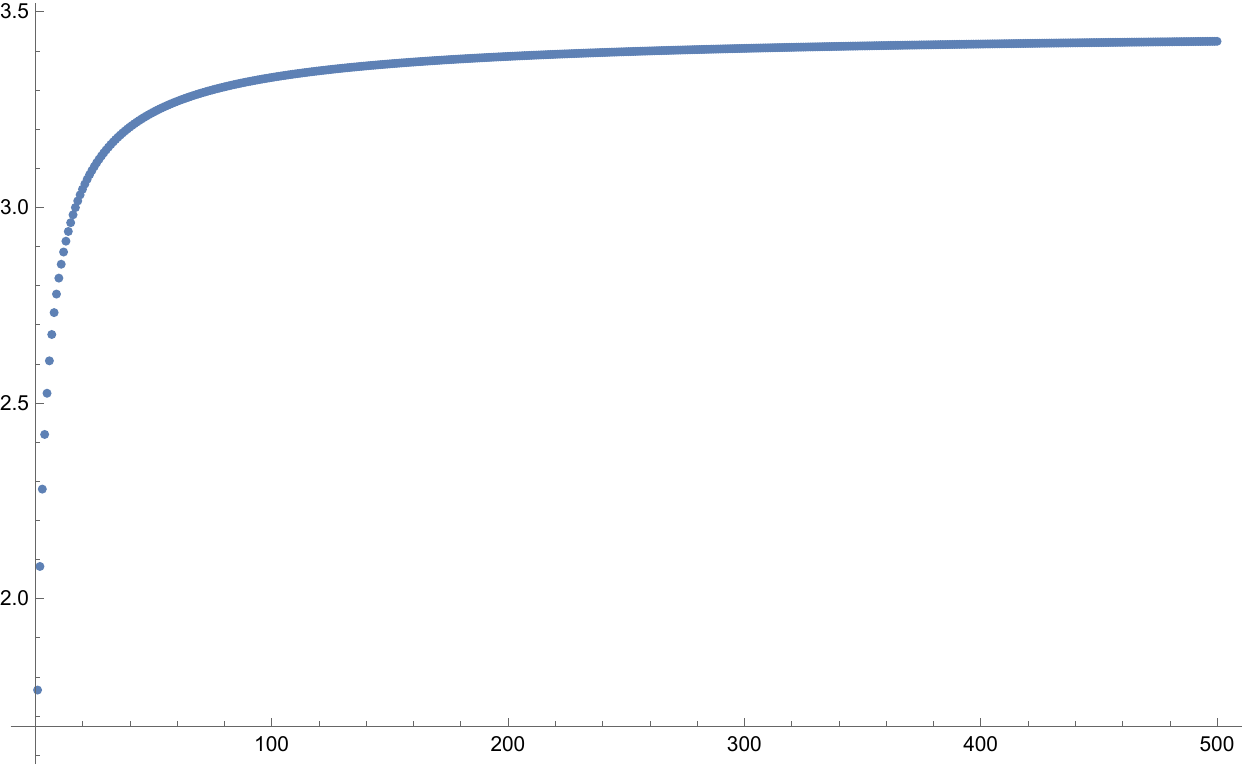}
  \caption{Convergence to the numerical radius}
  \label{fig:convplot}
\end{figure}

\newpage 
\bibliography{references}
\bibliographystyle{plain}


\end{document}